\documentclass[12pt]{amsart}

\usepackage{etex}
\usepackage{amsmath, amssymb}
\usepackage{array}
\usepackage[frame,cmtip,arrow,matrix,line,graph,curve]{xy}
\usepackage{graphpap, color, paralist, pstricks}
\usepackage[mathscr]{eucal}
\usepackage[pdftex]{graphicx}
\usepackage[pdftex,colorlinks,backref=page,citecolor=blue]{hyperref}
\usepackage{tikz}
\usepackage{tikz-cd}

\newtheorem{theorem}{Theorem}[section]
\newtheorem{proposition}[theorem]{Proposition}
\newtheorem{corollary}[theorem]{Corollary}
\newtheorem{lemma}[theorem]{Lemma}

\theoremstyle{definition}
\newtheorem{definition}[theorem]{Definition}
\newtheorem{example}[theorem]{Example}

\newtheorem{remark}[theorem]{Remark}

\renewcommand{\AA}{\mathbb{A} }
\newcommand{\CC}{\mathbb{C} }
\newcommand{\PP}{\mathbb{P} }

\newcommand{\RR}{\mathbb{R} }

\newcommand{\cS}{\mathcal{S} }

\def\GL{\mathrm{GL}}
\def\SL{\mathrm{SL}}

\def\Gr{\mathrm{Gr}}

\def\rk{\mathrm{rank}\, }
\def\im{\mathrm{im}\, }

\def\Trop{\mathrm{Trop}}
\def\trop{\mathrm{trop}}

\begin{document}

\title{Point configurations, phylogenetic trees, and dissimilarity vectors}
\keywords{tropical Grassmannian, dissimilarity vector, phylogenetic tree, rational normal curve}
\subjclass[2010]{05C05, 14M15, 14N10, 14T15}

\author{Alessio Caminata}
\address{Alessio Caminata, Dipartimento di Matematica, Universit\`a di Genova\\ via Dodecaneso 35, 16146, Genova, Italy}
\email{caminata@dima.unige.it}

\author{Noah Giansiracusa}
\address{Noah Giansiracusa, Department of Mathematical Sciences, Bentley University, Waltham, MA 02452, USA}
\email{ngiansiracusa@bentley.edu}

\author{Han-Bom Moon}
\address{Han-Bom Moon, Department of Mathematics, Fordham University, New York, NY 10023, USA}
\email{hmoon8@fordham.edu}

\author{Luca Schaffler}
\address{Luca Schaffler, Department of Mathematics, KTH Royal Institute of Technology, SE-100 44 Stockholm, Sweden}
\email{lucsch@math.kth.se}

\maketitle

\begin{abstract}
In 2004 Pachter and Speyer introduced the higher dissimilarity maps for phylogenetic trees and asked two important questions about their relation to the tropical Grassmannian.  Multiple authors, using independent methods, answered affirmatively the first of these questions, showing that dissimilarity vectors lie on the tropical Grassmannian, but the second question, whether the set of dissimilarity vectors forms a tropical subvariety, remained opened.  We resolve this question by showing that the tropical balancing condition fails.  However, by replacing the definition of the dissimilarity map with a weighted variant, we show that weighted dissimilarity vectors form a tropical subvariety of the tropical Grassmannian in exactly the way that Pachter--Speyer envisioned.  Moreover, we provide a geometric interpretation in terms of configurations of points on rational normal curves and construct a finite tropical basis that yields an explicit characterization of weighted dissimilarity vectors.
\end{abstract}

\section{Introduction}\label{sec:introduction}

\subsection{Background}\label{ssec:background}

In one of the first papers on tropical geometry, Speyer and Sturmfels \cite{SS04} introduced the tropical Grassmannian and showed that $\Gr^{\trop}(2,n)\subseteq \mathbb{R}^{\binom{n}{2}}$ coincides with the space of $n$-leaf phylogenetic trees, a tropical analogue of the moduli space of stable rational $n$-pointed curves that plays an important role in genomics.  With this Euclidean embedding, each phylogenetic tree is identified with its dissimilarity vector, the $\binom{n}{2}$-tuple of path lengths connecting each pair of the $n$ leaves.  

Pachter and Speyer \cite{PS04} generalized this embedding by introducing the higher dissimilarity maps: for each integer $r$ with $2 \le r \le \frac{n+1}{2}$ they showed that any phylogenetic tree can be recovered from its $r$-dissimilarity vector, the $\binom{n}{r}$-tuple recording the sum of edge lengths in the subtree spanned by each subset of $r$ leaves.  They also stated two questions concerning the possible tropical geometry of these higher dissimilarity maps: (1) Is the space of $r$-dissimilarity vectors in $\mathbb{R}^{\binom{n}{r}}$ contained in the tropical Grassmannian $\Gr^{\trop}(r,n)$? And if so, then: (2) Is there a rational map $\Gr(2,n) \dashrightarrow \Gr(r,n)$ whose image tropicalizes to yield the space of $r$-dissimilarity vectors? The first question was answered positively by several authors using distinct methods \cite{Coo09,Gir10,Man11}, whereas the second question has remained open other than the case $r=3$ that was confirmed in the original \cite{PS04}.  There have been numerous papers studying other aspects of Pachter--Speyer's higher dissimilarity maps as well (e.g., \cite{EL18, BR17, BR14, Rub12, HM12, BC09, Rub07, LYP06}).  

In this paper we resolve the second question of Pachter--Speyer and introduce and study a variant of the higher dissimilarity maps that is more compatible with tropical geometry.

\subsection{Statement of results}\label{ssec:results}

By direct calculation we provide a negative answer to the second tropical question of Pachter--Speyer (recall that the first open case is $r=4$, $n=7$):

\begin{theorem}[Theorem~\ref{prop:nonbal}]\label{thm:nottropicalvar}
For $n =7$ the space of $4$-dissimilarity vectors in $\mathbb{R}^{\binom{7}{4}}$ is a polyhedral complex that is not balanced, for any choice of weights on the facets, hence is not a tropical variety.
\end{theorem}

However, this is not the end of the story.  The rational map $\Gr(2,n)\dashrightarrow \Gr(3,n)$ in \cite{PS04}, providing the motivation for their second tropical question, does not tropicalize to a map sending the $2$-dissimilarity vector of each phylogenetic tree to the corresponding $3$-dissimilarity vectors --- as Pachter and Speyer point out, the output is \emph{twice} the corresponding $3$-dissimilarity vector.  This generalizes to a rational map $\Gr(2,n) \dashrightarrow \Gr(r,n)$ whose tropicalization sends the $2$-dissimilarity vector of a phylogenetic tree to the $\binom{n}{r}$-tuple recording, for each size $r$ subset of the $n$ leaves, the sum of all path lengths connecting all pairs of leaves in this subset. It is just a coincidence that for $r=3$ these two different notions of subtree weights differ by a scalar.  We call these $\binom{n}{r}$-tuples defined using path lengths within subtrees \emph{weighted} $r$-dissimilarity vectors, and the map sending a phylogenetic tree to its vector of weighted $r$-dissimilarity vectors the \emph{weighted} $r$-dissimilarity map.  While for $r>3$ the original $r$-dissimilarity vectors do not have the tropical geometry interpretation Pachter and Speyer had hoped for, it turns out these weighted variants do:

\begin{theorem}[Theorem~\ref{thm:tropim}, Proposition \ref{prop:GelMac}]\label{thm:weighteddissimilarity}
For $2 \le r \le n-2$, the weighted $r$-dissimilarity map embeds the space of phylogenetic trees as a tropical subvariety in $\mathbb{R}^{\binom{n}{r}}$.  This tropical variety is the tropicalization of a subvariety of $\Gr(r,n)$ that is both (1) the image of a natural rational map $\Gr(2,n) \dashrightarrow \Gr(r,n)$, and (2) the Gelfand--MacPherson correspondence applied to the open subvariety of $(\PP^{r-1})^n$ parameterizing configurations of $n$ distinct points that lie on a rational normal curve in $\PP^{r-1}$. 
\end{theorem}

The equations for the Zariski closure of the locus in $(\PP^{r-1})^n$ mentioned in the preceding theorem were studied in \cite{CGMS18}.  While they are not known in full generality, we prove here that a particularly simple subset of the defining equations, after applying the Gelfand--MacPherson correspondence, yields a tropical basis for an ideal whose set-theoretic vanishing locus is the subvariety of $\Gr(r,n)$ alluded to in the preceding theorem.  As a consequence of this tropical basis result, we obtain the following characterization of weighted dissimilarity vectors, generalizing the classic tree-metric theorem for $2$-dissimilarity vectors:

\begin{theorem}[Corollary \ref{cor:tropicalbasis}]\label{thm:tropicalbasisintro}
Fix  $2 \le r \le n-2$. A vector $w=(w_I)_{I\in \binom{[n]}{r}} \in \mathbb{R}^{\binom{n}{r}}$ is a weighted $r$-dissimilarity vector if and only if the following two conditions hold:
\begin{enumerate}

\item for each $4$-tuple $\{i,j,k,l\} \subseteq [n]$ there exists an $A\subseteq [n]\setminus \{i,j,k,l\}$ of size $r-2$ such that two of the following expressions equal each other and are greater than or equal to the third:
\[w_{ijA} + w_{klA},~w_{ikA}+w_{jlA},~w_{ilA} + w_{jkA};\]

\item for each $I \in \binom{[n]}{6}$, $J\in\binom{[n]\setminus I}{r-3}$, and for each cube $C$ on $I$ (see \S\ref{ssec:equationsVdn} for the notation) with corresponding bipartition $B,W$ we have

\[
\sum_{K \in B}w_{J\sqcup K}=\sum_{K \in W}w_{J\sqcup K}.
\]

\end{enumerate}
\end{theorem}

The case $r=2$ is a main result of \cite{SS04} and our proof relies on their result; in both this case and the case $r=n-2$ condition (2) here is vacuous because ${[n] \setminus I \choose r-3} = \emptyset$.  In general, this characterization does not provide a minimal, non-redundant set of conditions, and indeed our proof suggests an algorithmic approach for reducing the number of conditions of type (2) that need to be checked.

\begin{remark}
In \cite{SS04} it is shown that the quadratic Pl\"ucker relations do not form a tropical basis for $\Gr(r,n)$ when $r \ge 3$ and $n \ge 7$, and in general the tropical Grassmannian depends on the characteristic of the base field.  It is interesting to contrast with the present situation where the tropical subvariety of $\Gr^{\trop}(r,n)$ parameterizing weighted $r$-dissimilarity vectors, and the tropical basis we construct for it, is independent of the base field.
\end{remark}

\subsection*{Acknowledgements}
NG was supported in part by NSF DMS-1802263 and thanks the members of the Spring 2016 UGA VIGRE graduate student tropical research group: Natalie Hobson, Andrew Maurer, Xian Wu, Matt Zawodniak, and Nate Zbacnik. We also would like to thank the anonymous referee for the valuable comments and suggestions.

\section{Background and preliminaries}\label{sec:background}

First some conventions.  We work over an algebraically closed field $\Bbbk$ of arbitrary characteristic, equipped with the trivial valuation.  For a subvariety $X \subseteq \PP^{N-1}$ of projective space, we denote by $X^\circ := X^{\mathrm{aff}} \cap (\Bbbk^\times)^{N}$ the restriction of the affine cone over $X$ to the dense open torus in $\mathbb{A}^{N}$.  Tropicalization sends subvarieties of the torus $(\Bbbk^\times)^{N}$ to subsets of Euclidean space $\mathbb{R}^{N}$.

\subsection{Phylogenetic trees}\label{ssec:phylogenetictrees}

For us, an \emph{$n$-leaf phylogenetic tree} is a connected graph, without cycles or vertices of degree 2, with $n$ leaves labelled by the integers $[n] := \{1,\ldots,n\}$, that is equipped with an $\mathbb{R}$-valued length on each edge such that all the internal edges have non-negative length.  The set of $n$-leaf phylogenetic trees with a fixed combinatorial tree as the underlying graph forms a half-space $\mathbb{R}_{\ge 0}^{\#edges - n}\times \mathbb{R}^n$, and by identifying trees having edges of length zero with the trees obtained by deleting such edges these half-spaces are naturally glued together and form an abstract polyhedral complex, that we shall denote by $\mathcal{T}_n$, known as the \emph{space of phylogenetic trees} \cite{BHV01}.  An influential result of Speyer--Sturmfels is that the tropical Grassmannian 
\[
	\Gr^{\trop}(2,n) := \Trop(\Gr(2,n)^\circ) \subseteq \mathbb{R}^{\binom{n}{2}}
\] 
coincides with the space of phylogenetic trees $\mathcal{T}_n$ \cite[Theorem 3.4]{SS04}. 

\begin{remark}\label{rem:linearityspace}
A phylogenetic tree is sometimes defined to have edge lengths only on its internal edges.  The space of such phylogenetic trees is the quotient of $\Gr^{\trop}(2,n)$ by a linear subspace of dimension $n$, and it coincides with the moduli space of tropical $n$-pointed stable rational curves $\mathrm{M}^{\trop}_{0,n}$ somewhat analogous to Kapranov's construction \cite{Kap93} of $\overline{\mathrm{M}}_{0,n}$ as a (Chow) quotient of the Grassmannian $\Gr(2,n)$ by the maximal torus $(\Bbbk^{\times})^n$ (indeed the linear subspace $\mathbb{R}^n$ acting on $\Gr^{\trop}(2,n)$ is the tropicalization of Kapranov's torus action).  Throughout this paper we include the non-internal edge lengths and hence work in $\mathbb{R}^{\binom{n}{2}}$ without taking this linear subspace quotient.
\end{remark}

\subsection{Dissimilarity vectors and maps}

The map
\[d_2 : \mathcal{T}_n \rightarrow \mathbb{R}^{\binom{n}{2}},\]
sending each phylogenetic tree $T$ to the vector whose $(i < j)$-entry is the sum of edge lengths along the unique path in $T$ connecting leaf $i$ to leaf $j$ is known as the \emph{dissimilarity map}, and the output $d_2(T)$ is a \emph{dissimilarity vector}.  This map is injective \cite{Bun71}, with image equal to $\Gr^{\trop}(2,n)$; it identifies phylogenetic trees with dissimilarity vectors, or equivalently, points of the tropical Grassmannian \cite{SS04}.  

The \emph{higher} dissimilarity map 
\[d_r : \mathcal{T}_n \rightarrow \mathbb{R}^{\binom{n}{r}},\]
introduced in \cite{PS04} for $r \ge 3$, sends $T$ to the \emph{higher} dissimilarity vector whose $I$-entry, for $I\in\binom{[n]}{r}$, is the sum of edge lengths among all edges in the subtree spanned by the $r$ leaves indexed by $I$; it is injective for $2 \le r \le \frac{n+1}{2}$ \cite[Theorem in \S2]{PS04}.  Since a tree spanned by two leaves is a path, the $r=2$ case of this map coincides with the dissimilarity map in the preceding paragraph.  

Pachter and Speyer asked two questions about these higher dissimilarity maps \cite[Problems 3 and 4]{PS04}: (1) is the image of $d_r$ contained in $\Gr^{\trop}(r,n)$, and if so then (2) is there a rational map $\Gr(2,n) \dashrightarrow \Gr(r,n)$ whose image, viewed as a subvariety of $(\Bbbk^\times)^{\binom{n}{r}}$ by taking the affine cone over the Pl\"ucker embedding and then intersecting with the big torus in $\AA^{\binom{n}{r}}$, tropicalizes to the space of higher dissimilarity vectors $d_r(\mathcal{T}_n)$.  Various authors, cited above in the introduction, resolved the first of these questions in the affirmative.  For the second question there has been progress in characterizing the image of $d_r$ \cite{Rub12}, even in terms of a piecewise linear map that appears related to tropical geometry \cite{BC09}, but the only case that had been fully resolved is $r=3$, where in \cite[\S3]{PS04} it is observed that the rational map $\Gr(2,n) \dashrightarrow \Gr(3,n)$ induced by applying the second Veronese map to the columns of a $2\times n$ matrix achieves the desired goal.  This Pachter--Speyer map can be generalized as follows:

\begin{definition}\label{def:higherVeronese}
The matrix morphism $\AA^{2n} \to \AA^{rn}$,
\begin{equation}\label{eqn:psi}
	\left(\begin{array}{cccc}x_{1}&x_{2}& \cdots &x_{n}\\ y_{1} & y_{2}&\cdots & y_{n}\end{array}\right) 
	\mapsto 
	\left(\begin{array}{cccc}x_{1}^{r-1}&x_{2}^{r-1}& \cdots &x_{n}^{r-1}\\ x_{1}^{r-2}y_{1}&x_{2}^{r-2}y_{2} & \cdots &x_{n}^{r-2}y_{n}\\ x_{1}^{r-3}y_{1}^{2} & x_{2}^{r-3}y_{2}^{2} & \cdots & x_{n}^{r-3}y_{n}^{2}\\ &&\ddots \\
	y_{1}^{r-1} & y_{2}^{r-1}&\cdots &y_{n}^{r-1}\end{array}\right),
\end{equation}
given by applying the $(r-1)$-Veronese map to each column, descends to a rational map of Grassmannians $\Gr(2,n)\dashrightarrow \Gr(r,n)$ that we shall call the \emph{column-wise $(r-1)$-Veronese Grassmannian map}, or simply \emph{Veronese Grassmannian map} for short.
\end{definition}

The fact that this matrix map descends to the Grassmannians follows from the elementary observation that the image of each $\GL_2$-orbit is contained in a $\GL_r$-orbit.  Note, however, that the image of a full-rank matrix need not be a full-rank matrix, so at the level of Grassmannians this really is just a rational map and not a regular morphism; for instance, the full-rank matrix
\begin{displaymath}
\left( \begin{array}{ccccc}
1&0&0&\cdots&0\\
0&1&0&\cdots&0
\end{array} \right)
\end{displaymath}
is sent to the following non-full-rank matrix:
\begin{displaymath}
\left( \begin{array}{ccccc}
1&0&0&\cdots&0\\
0&0&0&\cdots&0\\
&&\vdots\\
0&1&0&\cdots&0
\end{array} \right).
\end{displaymath}
This column-wise Veronese Grassmannian map will play a central role in our paper.

\section{Resolution of Pachter--Speyer's second question}\label{sec:nonbal}
Recall that for each tree $G$ underlying an $n$-leaf phylogenetic tree (meaning $G$ has leaves labelled by $1,\ldots,n$ but the edges do not carry weights) there is a polyhedral cone in the space of phylogenetic trees $\mathcal{T}_n$, let us call it $\mathcal{T}_n^G$, parameterizing phylogenetic trees on $G$.  The restriction of the dissimilarity map $d_r$ to each such polyhedral cone is linear and so the image $d_r(\mathcal{T}_n^G)$ is a polyhedral cone in $\mathbb{R}^{\binom{n}{r}}$. By varying $G$, the polyhedral cones $d_r(\mathcal{T}_n^G)$ provide a polyhedral decomposition of the space of $r$-dissimilarity vectors.

For each edge $E$ of $G$, let $T_E$ be the phylogenetic tree on $G$ where $E$ has length $1$ and all the other edges have length $0$. In phylogenetics, such trees $T_E$ are called \emph{split metrics}. Then the polyhedral cone $d_r(\mathcal{T}_n^G)$ consists of all $\mathbb{R}$-linear combinations of the vectors $d_r(T_E)$ such that the coefficient on $d_r(T_E)$ is non-negative whenever $E$ is an internal edge.

To show that the space of $r$-dissimilarity vectors $d_r(\mathcal{T}_n)$ for $r>3$ is not the tropicalization of the image of a map $\Gr(2,n) \dashrightarrow \Gr(r,n)$, we show a stronger result: $d_r(\mathcal{T}_n)$ is not even a tropical variety in general. This is because, as a polyhedral complex, $d_r(\mathcal{T}_n)$ is not balanced for $r \ge 4$ (for the definition of balanced see \cite[Definition 3.3.1]{MS15}, and tropical varieties are balanced by \cite[Theorem 3.3.5]{MS15}). We check this explicitly in the first non-trivial case:

\begin{theorem}\label{prop:nonbal}
The $11$-dimensional polyhedral complex $d_4(\mathcal{T}_7) \subseteq \mathbb{R}^{\binom{7}{4}} = \mathbb{R}^{35}$ is not balanced for any choice of facet weights and hence it is not a tropical variety.
\end{theorem}

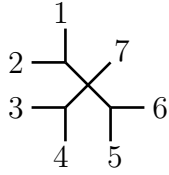
\begin{figure}
\begin{tikzpicture}[scale=0.3]

	\draw[line width=1pt] (0,0) -- (1,1);
	\draw[line width=1pt] (0,0) -- (-1,1);
	\draw[line width=1pt] (0,0) -- (1,-1);
	\draw[line width=1pt] (0,0) -- (-1,-1);
	\draw[line width=1pt] (-1,1) -- (-2.5,1);
	\draw[line width=1pt] (-1,1) -- (-1,2.5);
	\draw[line width=1pt] (1,-1) -- (2.5,-1);
	\draw[line width=1pt] (1,-1) -- (1,-2.5);
	\draw[line width=1pt] (-1,-1) -- (-2.5,-1);
	\draw[line width=1pt] (-1,-1) -- (-1,-2.5);

	\node at (-1.2,3.2) {$1$};
	\node at (-3.2,1) {$2$};
	\node at (-3.2,-1) {$3$};
	\node at (-1.2,-3.2) {$4$};
	\node at (1.2,-3.2) {$5$};
	\node at (3.2,-1) {$6$};
	\node at (1.5,1.5) {$7$};

\end{tikzpicture}
\caption{The graph $G$ defining a $10$-dimensional cone $\sigma$ in the space of $4$-dissimilarity vectors.}
\label{fig:treeexample}
\end{figure}

\begin{figure}

\begin{tikzpicture}[scale=0.25]

	\draw[line width=1pt] (0,-1) -- (0,1);
	\draw[line width=1pt] (0,1) -- (1,2);
	\draw[line width=1pt] (0,1) -- (-1,2);
	\draw[line width=1pt] (0,-1) -- (1,-2);
	\draw[line width=1pt] (0,-1) -- (-1,-2);
	\draw[line width=1pt] (-1,2) -- (-1,3.5);
	\draw[line width=1pt] (-1,2) -- (-2.5,2);
	\draw[line width=1pt] (1,-2) -- (1,-3.5);
	\draw[line width=1pt] (1,-2) -- (2.5,-2);
	\draw[line width=1pt] (-1,-2) -- (-1,-3.5);
	\draw[line width=1pt] (-1,-2) -- (-2.5,-2);

	\node at (1.2,-4.2) {$1$};
	\node at (3.2,-2) {$2$};
	\node at (-3.2,-2) {$3$};
	\node at (-1.2,-4.2) {$4$};
	\node at (-1.2,4.2) {$5$};
	\node at (-3.2,2) {$6$};
	\node at (1.5,2.5) {$7$};
	
	\node at (1.2,0.1) {$E_1$};

\end{tikzpicture}
\hspace{.2in}
\begin{tikzpicture}[scale=0.25]

	\draw[line width=1pt] (0,-1) -- (0,1);
	\draw[line width=1pt] (0,1) -- (1,2);
	\draw[line width=1pt] (0,1) -- (-1,2);
	\draw[line width=1pt] (0,-1) -- (1,-2);
	\draw[line width=1pt] (0,-1) -- (-1,-2);
	\draw[line width=1pt] (-1,2) -- (-1,3.5);
	\draw[line width=1pt] (-1,2) -- (-2.5,2);
	\draw[line width=1pt] (1,-2) -- (1,-3.5);
	\draw[line width=1pt] (1,-2) -- (2.5,-2);
	\draw[line width=1pt] (-1,-2) -- (-1,-3.5);
	\draw[line width=1pt] (-1,-2) -- (-2.5,-2);

	\node at (-3.2,-2) {$1$};
	\node at (-1.2,-4.2) {$2$};
	\node at (-1.2,4.2) {$3$};
	\node at (-3.2,2) {$4$};
	\node at (1.2,-4.2) {$5$};
	\node at (3.2,-2) {$6$};
	\node at (1.5,2.5) {$7$};
	
	\node at (1.2,0.1) {$E_2$};

\end{tikzpicture}
\hspace{.2in}
\begin{tikzpicture}[scale=0.25]

	\draw[line width=1pt] (0,-1) -- (0,1);
	\draw[line width=1pt] (0,1) -- (1,2);
	\draw[line width=1pt] (0,1) -- (-1,2);
	\draw[line width=1pt] (0,-1) -- (1,-2);
	\draw[line width=1pt] (0,-1) -- (-1,-2);
	\draw[line width=1pt] (-1,2) -- (-1,3.5);
	\draw[line width=1pt] (-1,2) -- (-2.5,2);
	\draw[line width=1pt] (1,-2) -- (1,-3.5);
	\draw[line width=1pt] (1,-2) -- (2.5,-2);
	\draw[line width=1pt] (-1,-2) -- (-1,-3.5);
	\draw[line width=1pt] (-1,-2) -- (-2.5,-2);

	\node at (-1.2,4.2) {$1$};
	\node at (-3.2,2) {$2$};
	\node at (-3.2,-2) {$3$};
	\node at (-1.2,-4.2) {$4$};
	\node at (1.2,-4.2) {$5$};
	\node at (3.2,-2) {$6$};
	\node at (1.5,2.5) {$7$};
	
	\node at (1.2,0.1) {$E_3$};

\end{tikzpicture}
\caption{The three graphs whose corresponding $11$-dimensional cones $\tau_1,\tau_2,\tau_3$ meet along the common face $\sigma$.}
\label{fig:degenerations}
\end{figure}
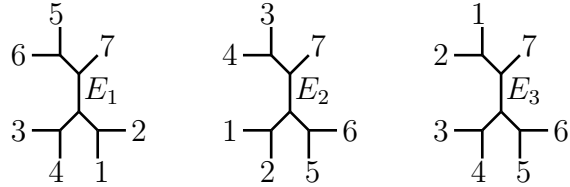

\begin{proof}
Consider the graph $G$ in Figure~\ref{fig:treeexample} with a unique vertex of degree $4$. This corresponds to a codimension-one cone $\sigma := d_4(\mathcal{T}_7^G)$ that is the common face of three maximal-dimensional cones, call them $\tau_1,\tau_2,\tau_3$. Each $\tau_i$ corresponds to the graph obtained by inserting an edge $E_i$ separating the $4$ incident edges in $G$ into two pairs of coincident edges (see Figure~\ref{fig:degenerations}).  Since the edge $E_i$ is internal, the cone $\tau_i$ is the $\mathbb{R}_{\ge 0}$-span of $\sigma$ and the vector $\rho_i := d_4(T_{E_i})$.  For $d_4(\mathcal{T}_7)$ to be balanced along $\sigma$ it is necessary that $\rho_1,\rho_2,\rho_3$ are linearly dependent modulo the subspace $\langle \sigma \rangle$ spanned by $\sigma$.  

Consider the $13\times 35$ matrix whose columns are indexed by the size 4 subsets of $[7]$ (ordered lexicographically: $\{1,2,3,4\}<\{1,2,3,5\}<\cdots<\{4,5,6,7\}$), whose first 10 rows are the images under $d_4$ of the split metrics defined by the edges in $G$, and whose last three rows are the images of the split metrics defined by the edges $E_i$:
{\footnotesize
\[
	\arraycolsep=1pt
	\left(\begin{array}{ccccccccccccccccccccccccccccccccccc}
1&1&1&1&1&1&1&1&1&1&1&1&1&1&1&1&1&1&1&1&0&0&0&0&0&0&0&0&0&0&0&0&0&0&0\\
1&1&1&1&1&1&1&1&1&1&0&0&0&0&0&0&0&0&0&0&1&1&1&1&1&1&1&1&1&1&0&0&0&0&0\\
1&1&1&1&0&0&0&0&0&0&1&1&1&1&1&1&0&0&0&0&1&1&1&1&1&1&0&0&0&0&1&1&1&1&0\\
1&0&0&0&1&1&1&0&0&0&1&1&1&0&0&0&1&1&1&0&1&1&1&0&0&0&1&1&1&0&1&1&1&0&1\\
0&1&0&0&1&0&0&1&1&0&1&0&0&1&1&0&1&1&0&1&1&0&0&1&1&0&1&1&0&1&1&1&0&1&1\\
0&0&1&0&0&1&0&1&0&1&0&1&0&1&0&1&1&0&1&1&0&1&0&1&0&1&1&0&1&1&1&0&1&1&1\\
0&0&0&1&0&0&1&0&1&1&0&0&1&0&1&1&0&1&1&1&0&0&1&0&1&1&0&1&1&1&0&1&1&1&1\\
0&1&1&0&1&1&0&1&1&1&1&1&0&1&1&1&1&1&1&1&1&1&0&1&1&1&1&1&1&1&1&1&1&1&1\\
1&1&1&1&1&1&1&0&0&0&1&1&1&1&1&1&1&1&1&0&1&1&1&1&1&1&1&1&1&0&1&1&1&1&1\\
1&1&1&1&1&1&1&1&1&1&1&1&1&1&1&1&1&1&1&1&1&1&1&1&1&1&1&1&1&1&0&0&0&0&0\\
0&1&1&1&1&1&1&1&1&1&1&1&1&1&1&1&1&1&1&1&1&1&1&1&1&1&1&1&1&1&1&1&1&1&1\\
1&1&1&1&1&1&1&0&1&1&1&1&1&1&1&1&1&1&1&1&1&1&1&1&1&1&1&1&1&1&1&1&1&1&1\\
1&1&1&1&1&1&1&1&1&1&1&1&1&1&1&1&1&1&1&1&1&1&1&1&1&1&1&1&1&1&0&1&1&1&1
	\end{array}\right)
\]
}
The last three rows are the vectors $\rho_i$ (the only entry of $\rho_i$ not equal to 1 is in the column indexed by the unique 4-tuple of vertices whose induced subgraph does not contain $E_i$).  With computer assistance we check that this matrix has full rank.  This implies that the last three rows are linearly independent modulo the subspace  $\langle \sigma \rangle$ spanned by the first $10$ rows, which shows that tropical balancing is not possible at $\sigma$.  
\end{proof}

\section{Weighted dissimilarity vectors}\label{sec:weighteddissimvectors}

In this section we tropicalize the Veronese Grassmannian map from Definition \ref{def:higherVeronese} and show that the image of the tropicalized map is the space of phylogenetic trees, embedded by the weighted dissimilarity vectors that we introduce in this paper.  One of the main steps is to recognize the Veronese Grassmannian map as the restriction of a monomial map of tori; this crucially avails us of functoriality of tropicalization.

\subsection{Coordinatizing the Veronese Grassmannian map}

Recall that the Veronese Grassmannian map $\Gr(2,n) \dashrightarrow\Gr(r,n)$ in Definition~\ref{def:higherVeronese} is expressed in terms of a matrix map of affine spaces $\mathbb{A}^{2n} \rightarrow \mathbb{A}^{rn}$.  In order to tropicalize it we need to coordinatize the induced map on Grassmannians in their Pl\"ucker embeddings.  Since these Grassmannians are obtained as $\GL$-quotients, this means expressing the matrix map in terms of homogeneous collections of $\SL$-invariants, i.e., maximal minors.   We do this by defining a morphism of tori $(\Bbbk^\times)^{\binom{n}{2}}\rightarrow(\Bbbk^\times)^{\binom{n}{r}}$ that restricts to the Veronese Grassmannian map.

\begin{remark} 
Since technically in this paper we tropicalize projective varieties by first lifting to affine cones and then restricting to dense tori, by a slight abuse of terminology we shall use the term \emph{Veronese Grassmannian map} to refer to the rational map $\Gr(2,n) \dashrightarrow\Gr(r,n)$ in Definition~\ref{def:higherVeronese} as well as the induced morphism $\Gr(2,n)^\circ \rightarrow \Gr(r,n)^\circ$ (the fact that the latter is indeed a regular morphism follows from Proposition \ref{prop:restrict} below); the context will always make clear which meaning is intended.  
\end{remark}

\begin{definition}
\label{definitionofphir}
For each $2 \le r \le n$, let 
\[\varphi_r : (\Bbbk^\times)^{{n \choose 2}} \to (\Bbbk^\times)^{{n \choose r}}\] be the group scheme morphism induced from the $\Bbbk$-algebra homomorphism
\begin{eqnarray*}
	\varphi_r^{*} : \Bbbk[x_{I}^{\pm}] &\to& \Bbbk[x_{ij}^{\pm}]\\
	x_{I} &\mapsto & \prod_{i,j \in I, i < j}x_{ij}.
\end{eqnarray*}
\end{definition}

\begin{proposition}\label{prop:restrict}
The monomial morphism $\varphi_r$ restricts to the Veronese Grassmannian map \[\Gr(2,n)^\circ \rightarrow \Gr(r,n)^\circ.\]
\end{proposition}

\begin{proof}
By the first fundamental theorem of invariant theory, we need to see how the maximal minors of the right-hand matrix in \eqref{eqn:psi} depend on the maximal minors of the left-hand matrix.  But the right-hand matrix is just a Vandermonde matrix where the columns have been homogenized, so for any collection  $I \in\binom{[n]}{r}$ of columns the corresponding maximal minor is \[\prod_{i,j\in I, i < j}(x_iy_j - x_jy_i) = \prod_{i,j \in I, i < j}m_{ij},\] where $m_{ij}$ denotes the $ij$-maximal minor of the left-hand matrix.  This shows the restricted morphism $\varphi_r|_{\Gr(2,n)^\circ}$ is indeed induced by the column-wise Veronese map.
\end{proof}

Since $\varphi_r$ is a toric morphism, we can now apply functoriality of tropicalization for toric morphisms \cite[Corollary~3.2.13]{MS15} which tells us that the tropicalization of the closure (in $(\Bbbk^\times)^{\binom{n}{r}}$) of the image of the Veronese Grassmannian map coincides with the image of the tropicalized map $\Trop(\varphi_r)$ restricted to the tropical Grassmannian $\Gr^{\trop}(2,n)$.  As we discussed earlier, $\Gr^{\trop}(2,n)$ is the space of $2$-dissimilarity vectors, and $\Trop(\varphi_r)$ is the linear map described explicitly in the following proposition (whose proof is trivial).  Our next steps are to go through this functoriality argument in detail, and to interpret $\Trop(\varphi_r)$ as sending $2$-dissimilarity vectors to the weighted dissimilarity vectors that we introduce next.  

\begin{proposition}\label{prop:tropphi}
The monomial morphism \[\varphi_r : (\Bbbk^\times)^{{n \choose 2}} \to (\Bbbk^\times)^{{n \choose r}}\] tropicalizes to the linear map \[\Trop(\varphi_r) : \mathbb{R}^{\binom{n}{2}} \to \mathbb{R}^{\binom{n}{r}}\] whose $I$-component, for $I\in\binom{[n]}{r}$, is $\sum_{i,j\in I,i<j} x_{ij}$.
\end{proposition}

\subsection{Weighted dissimilarity maps}

\begin{definition}
For each $2 \le r \le n$, let 
\[d^{wt}_r : \mathcal{T}_n \rightarrow \mathbb{R}^{\binom{n}{r}}\]
be the \emph{weighted dissimilarity map} sending a phylogenetic tree $T$ to the \emph{weighted dissimilarity vector} defined as follows.  For each $I \in \binom{[n]}{r}$, let $T(I)$ be the $r$-leaf subtree of $T$ spanned by the leaves indexed by $I$, and let the $I$-component of $d^{wt}_r(T)$ be the sum of the entries of the dissimilarity vector $d_2(T(I))$.
\end{definition}

In other words, $d^{wt}_r$ records for each $r$-leaf subtree the sum of all $\binom{r}{2}$ path lengths in the subtree.  The usual $r$-dissimilarity map $d_r$ records for each $r$-leaf subtree the sum of all edge lengths in the subtree, whereas $d^{wt}_r$ is a ``weighted'' variant because it counts each edge with multiplicity equal to the number of leaf-to-leaf paths in the subtree in which the edge occurs. 

\begin{remark}
Note that $d^{wt}_3 = 2d_3$ since in a $3$-leaf tree every edge is traversed exactly twice among the $\binom{3}{2}=3$ possible leaf-to-leaf paths, whereas for $r>3$ the usual and weighted dissimilarity maps are, in general, not simply scalar multiples of each other.
\end{remark}

We will later show that the image of the weighted dissimilarity map is a tropical variety (Theorem~\ref{thm:tropim}) and hence in particular is a balanced polyhedral complex.  Before getting to that general proof, one might be curious to see how the matrix used to establish non-balancing in the proof of Theorem~\ref{prop:nonbal} changes when using the weighted dissimilarity map.

\begin{example}
Replacing every instance of the dissimilarity map $d_4$ with the weighted dissimilarity map $d^{wt}_4$ in the construction of the $13\times35$ matrix in the proof of Theorem~\ref{prop:nonbal} yields the following:
{\footnotesize
\[
	\arraycolsep=1pt
	\left(\begin{array}{ccccccccccccccccccccccccccccccccccc}
3&3&3&3&3&3&3&3&3&3&3&3&3&3&3&3&3&3&3&3&0&0&0&0&0&0&0&0&0&0&0&0&0&0&0\\
3&3&3&3&3&3&3&3&3&3&0&0&0&0&0&0&0&0&0&0&3&3&3&3&3&3&3&3&3&3&0&0&0&0&0\\
3&3&3&3&0&0&0&0&0&0&3&3&3&3&3&3&0&0&0&0&3&3&3&3&3&3&0&0&0&0&3&3&3&3&0\\
3&0&0&0&3&3&3&0&0&0&3&3&3&0&0&0&3&3&3&0&3&3&3&0&0&0&3&3&3&0&3&3&3&0&3\\
0&3&0&0&3&0&0&3&3&0&3&0&0&3&3&0&3&3&0&3&3&0&0&3&3&0&3&3&0&3&3&3&0&3&3\\
0&0&3&0&0&3&0&3&0&3&0&3&0&3&0&3&3&0&3&3&0&3&0&3&0&3&3&0&3&3&3&0&3&3&3\\
0&0&0&3&0&0&3&0&3&3&0&0&3&0&3&3&0&3&3&3&0&0&3&0&3&3&0&3&3&3&0&3&3&3&3\\
0&3&3&0&3&3&0&4&3&3&3&3&0&4&3&3&4&3&3&4&3&3&0&4&3&3&4&3&3&4&4&3&3&4&4\\
4&3&3&3&3&3&3&0&0&0&4&4&4&3&3&3&3&3&3&0&4&4&4&3&3&3&3&3&3&0&4&4&4&3&3\\
4&4&4&4&4&4&4&4&4&4&3&3&3&3&3&3&3&3&3&3&3&3&3&3&3&3&3&3&3&3&0&0&0&0&0\\
0&3&3&3&3&3&3&4&4&4&3&3&3&4&4&4&4&4&4&3&3&3&3&4&4&4&4&4&4&3&4&4&4&3&3\\
4&3&3&4&3&3&4&0&3&3&4&4&3&3&4&4&3&4&4&3&4&4&3&3&4&4&3&4&4&3&4&3&3&4&4\\
4&4&4&3&4&4&3&4&3&3&3&3&4&3&4&4&3&4&4&4&3&3&4&3&4&4&3&4&4&4&0&3&3&3&3
	\end{array}\right).
\]
}
Recall that the first $10$ rows are the images of the split metrics defined by the graph $G$ in Figure~\ref{fig:treeexample}, and the last $3$ rows are the images of the split metrics defined by the edges $E_i$ in Figure~\ref{fig:degenerations}.   Note that $d^{wt}_4$ and $d_4$ indeed are not scalar multiples of each other, but as expected the locations of the zero entries in this matrix are the same as in the previous matrix.  For this matrix, the first $10$ rows are linearly independent but the whole matrix has rank $12$. Hence, the last three rows are linearly dependent modulo the linear subspace spanned by the previous $10$, and this is what allows for balancing to hold here.  Explicitly, the one-dimensional left kernel is spanned by the vector
\[
(0,0,0,0,0,0,1,1,1,1,-1,-1,-1),
\]
which tells us that the sum of the images of the split metrics given by the $3$ edges $E_i$ equals the sum of the images of the split metrics given by the $4$ coincident edges in the graph $G$.
\end{example}

The following proposition, whose proof follows immediately from the definition and Proposition~\ref{prop:tropphi}, plays a fundamental role in this paper (indeed, we were led to the definition of the weighted dissimilarity map primarily so that this holds):

\begin{proposition}\label{prop:compo}
The weighted dissimilarity map factors as follows:
\[d_r^{wt} = \Trop(\varphi_r)\circ d_2.\]
\end{proposition}

Accordingly, in order to better understand the weighted dissimilarity map, we need to first establish a key property of the linear map $\Trop(\varphi_r)$.

\begin{lemma}\label{lem:phiinjective}
For $r \le n-2$ the map $\Trop(\varphi_r)$ is injective, and for $r \in \{2,n-2\}$ it is bijective. 
\end{lemma}

\begin{proof}
This is trivial for $r=2$, since $\varphi_2$ is the identity map, so assume $r \ge 3$. Let $M$ be the matrix associated to $\Trop(\varphi_r)$, namely: 
\begin{equation}\label{eqn:matrixM}
	M_{IJ} = \begin{cases}1 & \mbox{if}~J \subseteq I\\
	0 & \mbox{otherwise}.\end{cases}
\end{equation}
We will construct an explicit left-inverse of $M$.  Define the ${n \choose 2} \times {n \choose r}$ matrix $M^{+}$ by 
\[
	M^{+}_{JI} = (-1)^{i}\frac{r-2}{r-i}\cdot \frac{1}{{n-2 \choose r-i}}, 
\]
where $i = |I \cap J|$. That is, for $J \in {[n] \choose 2}$ and $I \in {[n] \choose r}$ we have 
\[
	M^{+}_{JI} = \begin{cases}
	\frac{1}{{n-2\choose r-2}} & \mbox{if}~J\subseteq I\\
	-\frac{r-2}{r-1}\cdot \frac{1}{{n-2 \choose r-1}} & \mbox{if}~|J \cap I| = 1\\
	\frac{r-2}{r}\cdot \frac{1}{{n-2 \choose r}} & \mbox{if}~J \cap I = \emptyset.
	\end{cases}
\]
We will show that $M^{+}M = \mathrm{Id}$ by directly calculating its entries. First of all, 
\[
\	(M^{+}M)_{JJ} = \sum_{I}M_{JI}^{+}M_{IJ} = \sum_{I \supset J}M_{JI}^{+} = \sum_{I \supset J}\frac{1}{{n-2 \choose r-2}} = \frac{1}{{n-2 \choose r-2}}{n-2 \choose r-2} = 1.
\]
For $J, K \in {[n] \choose 2}$, we have 
\[
	(M^{+}M)_{JK} = \sum_{I}M_{JI}^{+}M_{IK}\\
	= \sum_{I \supset J, K}\frac{1}{{n-2 \choose r-2}} - \sum_{|I \cap J| = 1, I \supset K}\frac{r-2}{r-1}\cdot \frac{1}{{n-2 \choose r-1}} + \sum_{I \cap J = \emptyset, I \supset K}\frac{r-2}{r}\cdot \frac{1}{{n-2 \choose r}}.
\]
If $|J \cap K| = 1$, then the condition in the third summation is impossible, since $J \cap K \ne \emptyset$, so
\[
	(M^{+}M)_{JK} = \frac{{n-3 \choose r-3}}{{n-2 \choose r-2}} - \frac{r-2}{r-1}\cdot\frac{{n-3 \choose r-2}}{{n-2\choose r-1}} = 0,
\]
where the last equality follows from an elementary calculation.  If instead $J \cap K = \emptyset$ then 
\[
	(M^{+}M)_{JK} = \frac{{n-4 \choose r-4}}{{n-2 \choose r-2}} - \frac{r-2}{r-1}\cdot \frac{2{n-4 \choose r-3}}{{n-2 \choose r-1}} + \frac{r-2}{r}\cdot \frac{{n-4 \choose r-2}}{{n-2 \choose r}} = 0,
\]
where again the last equality is an elementary calculation.

The equality $\dim \RR^{{n \choose 2}} = \dim \RR^{{n \choose n-2}}$ then implies surjectivity when $r=2$ or $r=n-2$. 
\end{proof}

\begin{remark}
By Proposition~\ref{prop:compo}, the matrix $M^{+}$ constructed in the preceding proof, when viewed as a linear map $\mathbb{R}^{\binom{n}{r}} \rightarrow \mathbb{R}^{\binom{n}{2}}$, sends the weighted $r$-dissimilarity vector of a phylogenetic tree to the corresponding $2$-dissimilarity vector.
\end{remark}

\begin{corollary}\label{cor:injectivedrwt}
For $r \le n-2$, the weighted dissimilarity map $d_r^{wt} : \mathcal{T}_n \rightarrow \mathbb{R}^{\binom{n}{r}}$ is injective.
\end{corollary}

\begin{proof}
Lemma~\ref{lem:phiinjective} and Proposition~\ref{prop:compo}, together with the fact that the $2$-dissimilarity map is injective, show that $d_r^{wt}$ is a composition of injective maps, and hence is injective.  
\end{proof}

\begin{corollary}\label{cor:identification}
For $r \le n-2$, the space of phylogenetic trees $\mathcal{T}_n$ and the space of weighted $r$-dissimilarity vectors are isomorphic as combinatorial polyhedral complexes. Furthermore, if $r \le \frac{n+1}{2}$ then they are also isomorphic to the space of $r$-dissimilarity vectors.
\end{corollary}

\begin{proof}
This follows from the injectivity of $d_{r}^{wt}$ in Corollary~\ref{cor:injectivedrwt}, the additional injectivity of $d_{r}$ when $r \le \frac{n+1}{2}$, and the observation that both maps are linear on each polyhedral stratum of $\mathcal{T}_n$.
\end{proof}

Although both the dissimilarity map and the weighted dissimilarity map provide Euclidean embeddings of the space of phylogenetic trees, we have seen in \S\ref{sec:nonbal} that the former embedding is not a tropical variety; we show in the following subsection that the latter embedding is tropical and we use the Veronese Grassmannian map to produce an algebraic variety realizing it as a tropicalization.

\subsection{Back to Pachter--Speyer's second question}\label{ssec:weightedvariant}

Recall that the second question of Pachter--Speyer, whether the space of $r$-dissimilarity vectors is the tropicalization of the image of a rational map of Grassmannians, ended up being false for the plain reason that higher dissimilarity vectors are not a balanced polyhedral complex and hence cannot be a tropical variety.  We now establish a positive answer to the variant of Pachter--Speyer's second question where dissimilarity vectors are replaced with weighted dissimilarity vectors:

\begin{theorem}\label{thm:tropim}
For $r \le n$, the space of weighted $r$-dissimilarity vectors is the tropicalization of the image of the Veronese Grassmannian map $\Gr(2,n)^\circ \rightarrow \Gr(r,n)^\circ$. 
\end{theorem}

\begin{proof}
By functoriality of tropicalization with respect to toric morphisms \cite[Corollary~3.2.13]{MS15}, we have that
\[
\Trop(\varphi_r)\left(\Gr^{\trop}(2,n)\right)=\Trop\left(\overline{\varphi_r(\Gr(2,n)^\circ)}\right).
\]
By Proposition \ref{prop:restrict}, $\varphi_r(\Gr(2,n)^\circ)$ is the image of the Veronese Grassmannian map; by Lemma~\ref{lem:closed}, below, this image is closed in the torus so we can ignore the Zariski closure in the right-hand side of this equality; by Proposition~\ref{prop:compo}, the left-hand side is $d_r^{wt}(\mathcal{T}_n)$. 
\end{proof}

\begin{lemma}\label{lem:closed}
For $r \le n$, the image $\varphi_r(\Gr(2,n)^\circ)$ is closed in $(\Bbbk^\times)^{{n \choose r}}$.
\end{lemma}

\begin{proof}
Let $x\in \overline{\varphi_r(\Gr(2,n)^\circ)}\subseteq(\Bbbk^\times)^{\binom{n}{r}}$, and let $R$ be a DVR with field of fractions $K$ and residue field $\Bbbk$ such that we have a map $\mathrm{Spec}(R)\rightarrow \overline{\varphi_r(\Gr(2,n)^\circ)}$ with $\mathrm{Spec}(K)$ mapping to $\varphi_r(\Gr(2,n)^\circ)$ and $\mathrm{Spec}(\Bbbk)$ mapping to $x$.  Let $U\subseteq \mathbb{A}^{rn}$ be the open locus of matrices all of whose maximal minors are nonzero.  The $\SL_r$-quotient morphism $U \rightarrow \Gr(r,n)^\circ$ is a locally trivial bundle in the Zariski topology, so we can lift $\mathrm{Spec}(R)\rightarrow \overline{\varphi_r(\Gr(2,n)^\circ)}$ to a map $\mathrm{Spec}(R)\rightarrow U$; fix a choice of lift. This is a matrix over $R$ all of whose maximal minors are nonzero --- so in particular none of the columns of this matrix is the zero vector --- and whose restriction to $\mathrm{Spec}(K)$ is, up to the $\SL_r$-action, a matrix in the form shown in the right-hand side of \eqref{eqn:psi}.  

Because none of the columns of this matrix is zero, it descends to an $R$-point of the $(\Bbbk^\times)^n$-quotient $(\PP_R^{r-1})^n$.  The restriction of this latter $R$-point to $\mathrm{Spec}(K)$ is a configuration of $n$ points in $\PP^{r-1}_K$ that lie on a rational normal curve, because the map in \eqref{eqn:psi} simply applies the $(r-1)$-Veronese map to each column and the $\SL_r$-action preserves the property of the configuration lying on a rational normal curve.  Therefore, the induced $\Bbbk$-point is in the Zariski closure of the locus of $n$ points lying on a rational normal curve, and it is non-degenerate by the non-vanishing of maximal minors. So by \cite[Proposition 2.7]{CGMS18} this $\Bbbk$-point is a configuration of $n$ points on a quasi-Veronese curve (a non-degenerate flat limit of rational normal curves, see \cite[Definition 2.5]{CGMS18}), which we denote by $C \subseteq \PP^{r-1}$. We claim there is an actual rational normal curve $C' \subseteq \PP^{r-1}$ containing this $n$-point configuration. 

If $C$ is irreducible then it is a rational normal curve and we may set $C' = C$. Suppose not, i.e., $C$ is a reducible quasi-Veronese curve.  We can then write $C = C_{1} \cup C_{2}$ where, by \cite[Lemma 2.6]{CGMS18}, $C_{1}$ and $C_{2}$ are connected, possibly reducible, quasi-Veronese curves of positive degrees $d_{1}$ and $d_{2}$, respectively, with $d_{1} + d_{2} = r-1$. The same lemma shows that the projective linear subspace spanned by a degree $d_{i}$ quasi-Veronese curve is of dimension $d_{i}$. It follows that the number of points lying on $C_{i}$ is at most $d_{i}+1$, for $i=1,2$, since otherwise the points on $C_i$ would be linearly dependent and so any set of $r$ points containing these points would also be linearly dependent, contradicting the fact that all maximal minors of the corresponding matrix are nonzero.  Consequently, 
\[n \le d_{1}+d_{2}+2 = r+1.\] Thus we have at most $r+1$ points in $\PP^{r-1}$, and they are in general linear position by the nonzero maximal minors condition, so Castelnuovo's lemma (\cite[Theorem 1.18]{Har95}) implies the existence of a rational normal curve $C'$ through all $n$ points, as claimed. 

Any rational normal curve in $\mathbb{P}^{r-1}$ is in the $\GL_r$-orbit of the standard Veronese rational normal curve $\mathbb{P}^1\hookrightarrow\mathbb{P}^{r-1}$. So, up to acting on the lift $\mathrm{Spec}(R)\rightarrow U$ by $\SL_r$, we can assume that $C'$ is the standard Veronese rational normal curve. This implies that the corresponding limiting $\Bbbk$-point in $U$ is in the form shown in the right-hand side of \eqref{eqn:psi}, so its image $x$ under the $\SL_r$-quotient $U \rightarrow \Gr(r,n)^\circ$ is indeed in the image of $\varphi_r$. 
\end{proof}

\begin{remark}
\label{boccicoolswork}
In \cite[Theorem~3.2]{BC09} Bocci--Cools introduce a piecewise linear map \[\phi^{(r)}:\mathbb{R}^{n\choose2}\rightarrow\mathbb{R}^{n\choose r}\] that provides a factorization of the $r$-dissimilarity map, namely $d_r = \phi^{(r)} \circ d_2$.  On the other hand, as shown in Proposition \ref{prop:compo} our linear map $\Trop(\varphi_r)$ provides a factorization of our weighted $r$-dissimilarity map, namely $d_r^{wt} = \Trop(\varphi_r)\circ d_2$.  Since $\Trop(\varphi_r)$ is injective, we can choose a left inverse for it (such as the one explicitly constructed in the proof of Lemma \ref{lem:phiinjective}) and then composing this with $\phi^{(r)}$ yields a piecewise linear map $g_r\colon\mathbb{R}^{n\choose r}\rightarrow\mathbb{R}^{n\choose r}$ such that the following diagram commutes:
\begin{equation*}
\begin{tikzcd}
&\mathcal{T}_n\arrow{r}{d_2}\ar[drr, "d_r~" left]\ar[rr, bend left, "d_r^{wt}"]
&\mathbb{R}^{n\choose2} \arrow{r}{\Trop(\varphi_r)}\arrow{dr}{\phi^{(r)}}
&\mathbb{R}^{n\choose r} \arrow{d}{g_r}\\
&
&
&\mathbb{R}^{n\choose r}.
\end{tikzcd}
\end{equation*}
In particular, we obtain a factorization $d_{r} = g_r \circ d_{r}^{wt}$.  As we have seen, $d_r^{wt}(\mathcal{T}_n)$ is a tropical variety in $\mathbb{R}^{\binom{n}{r}}$ whereas $d_r(\mathcal{T}_n)$ is not.  Intuitively, the map $g_r$ tilts rays in the space of weighted dissimilarity vectors in such a way that certain collections of rays go from being linearly dependent to being linearly independent, and this is what destroys the balancing condition needed to be a tropical variety.
\end{remark}

\section{Tropical bases and a generalized tree-metric theorem}

Recall that $\varphi_r(\Gr(2,n)^\circ) \subseteq \Gr(r,n)^\circ$ is a closed subvariety (in the ambient torus $(\Bbbk^\times)^{\binom{n}{r}}$) whose tropicalization is the space of weighted $r$-dissimilarity vectors $d_r^{wt}(\mathcal{T}_n) \subseteq \mathbb{R}^{\binom{n}{r}}$.  In order to find tropical equations for the tropicalization of this subvariety --- and hence a characterization of weighted dissimilarity vectors --- we need to first find equations for the subvariety $\varphi_r(\Gr(2,n)^\circ)$ itself.

\subsection{Gelfand--MacPherson correspondence}

The proof of Lemma~\ref{lem:closed} shows that points of $\varphi_r(\Gr(2,n)^\circ)$ correspond to configurations of $n$ points in $\PP^{r-1}$ that lie on a rational normal curve.  This correspondence is in essence the Gelfand--MacPherson correspondence, which identifies generic $\GL_r$-orbits in $(\PP^{r-1})^n$ with generic $(\Bbbk^\times)^n$-orbits in $\Gr(r,n)$, and vice-versa (cf. \cite[\S2.2]{Kap93}).  In fact:

\begin{proposition}\label{prop:GelMac}
For $r \le n$, $\varphi_r(\Gr(2,n)^\circ)$ corresponds under Gelfand--MacPherson to the open locus in $(\PP^{r-1})^n$ of configurations of $n$ distinct points that lie on a rational normal curve.
\end{proposition}

\begin{proof}
The proof of Lemma~\ref{lem:closed} shows that each point of $\varphi_r(\Gr(2,n)^\circ)$ corresponds to a configuration of $n$ points on a rational normal curve, and these points must be distinct since otherwise two columns in the matrix of coordinates would be proportional and hence any maximal minor containing these columns would be zero, contradicting the fact that all maximal minors are nonzero.  Conversely, it is a classical fact (coming from the Vandermonde determinant) that distinct points on a rational normal curve are linearly independent, so any configuration of such points corresponds to a matrix all of whose maximal minors are nonzero, and as noted in the proof of Lemma~\ref{lem:closed} such a matrix yields a point of $\varphi_r(\Gr(2,n)^\circ)$.
\end{proof}

In particular, any $\SL_r$-invariant polynomial that vanishes on the locus in $(\PP^{r-1})^n$ of configurations lying on a rational normal curve corresponds to a $(\Bbbk^\times)^n$-invariant polynomial that vanishes on $\varphi_r(\Gr(2,n)^\circ)$.    In other words, to find the ideal defining $\varphi_r(\Gr(2,n)^\circ)$, a natural place to look is the ideal defining the Zariski closure in $(\PP^{r-1})^n$ of the locus of points lying on a rational normal curve.  This latter closed subvariety, and the ideal defining it, was the focus of the paper \cite{CGMS18}, where it is denoted $V_{r-1,n} \subseteq (\PP^{r-1})^n$ (since it parameterizes configurations on a quasi-Veronese curve).  

Two potential issues arise with this strategy: (1) generators for the ideal of $V_{r-1,n}$ are not fully known in general, and (2) not all the generators for this ideal are $\SL_r$-invariant \cite[Remark 4.24]{CGMS18}.  However, we will establish in this section that the generators that are known from \cite{CGMS18} (all of which are $\SL_r$-invariant) suffice to cut out the tropicalization of $\varphi_r(\Gr(2,n)^\circ)$.  We begin by reviewing these equations.

\subsection{Equations for points to lie on a rational normal curve}
\label{ssec:equationsVdn}

The closure $V_{r-1,n} \subseteq (\PP^{r-1})^n$ of the locus of $n$ points lying on a rational normal curve in $\PP^{r-1}$ is the whole space if $r=2$ or $r\ge n-2$. Thus, we will assume $3\leq r\leq n-3$ from now on. The first nontrivial example of  $V_{r-1,n} \subseteq (\PP^{r-1})^n$ is $V_{2, 6}$, which parametrizes six-tuples of points in $\mathbb{P}^2$ that lie on a conic. This is an irreducible hypersurface in $(\mathbb{P}^2)^6$ defined by the vanishing of the following $\SL_3$-invariant polynomial expressed as a quartic binomial in bracket notation (see \cite[Equation~(3.4.9)]{Stu08} and \cite[Remark~3.3]{CGMS18}):
\[
	\phi = |123||145||246||356| - |124||135||236||456|.
\]
The notation $|ijk|$ here denotes the determinant of the $3\times 3$ submatrix, of a $3\times 6$ matrix of coordinates on $(\PP^2)^6$, with columns $ijk$.  This bracket expression is not fully $S_{6}$-symmetric because brackets satisfy many non-trivial Pl\"ucker relations. Indeed, up to obvious sign changes there are 15 different presentations of $\phi$, as we next describe. 

Let $G$ be the graph with vertex set ${[6] \choose 3}$ where vertices $I$ and $J$ are connected if $|I \cap J| = 2$. A  straightforward combinatorial argument shows that $G$ has 15 subgraphs isomorphic to the $3$-dimensional cube, and these form a single orbit under the natural $S_6$-action.  A cube is a bipartite graph, so for each cube subgraph we can uniquely divide the vertex set into black and white subsets, which we label $B$ and $W$ respectively, where we adopt the convention that the smallest triplet in the lexicographic order is black.  For each vertex $I = \{i, j, k\}$ we have the associated bracket $m_{I} := |ijk|$, and for each cube $C$ in $G$ we may define a polynomial 
\[
	\phi_{C} := \prod_{I \in B}m_{I} - \prod_{J \in W}m_{J}.
\]

\begin{example}\label{example:cube}
The subgraph $C$ generated by 
	\[\{1,2,3\}, \{1, 2, 4\}, \{1, 3, 5\}, \{1, 4, 5\}, \{2, 3, 6\}, \{2, 4, 6\}, \{3, 5, 6\}, \{4, 5, 6\}\]
is a cube, and the corresponding black-white bipartition is
\[
	B := \{\{1, 2, 3\}, \{1, 4, 5\}, \{2, 4, 6\}, \{3, 5, 6\}\}, \quad 
	W := \{\{1, 2, 4\}, \{1, 3, 5\}, \{2, 3, 6\}, \{4, 5, 6\}\},
\]
so in this case $\phi_C$ coincides with the polynomial $\phi$ presented above.
\end{example}

\begin{lemma}
For each cube $C$, we have $V(\phi_C) = V(\phi)$ as subvarieties of $(\PP^2)^6$. 
\end{lemma}

\begin{proof}
As noted above, $\phi = \phi_C$ where $C$ is the cube in Example \ref{example:cube}, so it suffices to show that if $C'$ is another cube then $V(\phi_{C}) = V(\phi_{C'})$.  By geometric considerations, the irreducible hypersurface $V_{2,6} = V(\phi_C)$ is invariant under the natural $S_6$-action on $(\PP^2)^6$.  This implies that any $S_6$-permutation of $\phi_C$ must be a polynomial whose vanishing locus is also $V_{2,6}$.  The transitive $S_{6}$-action on the set of cubes is compatible with the action on bracket polynomials induced from the permutation action on $(\PP^{2})^{6}$.  Therefore, for any cube $C'$ there exists a permutation $\sigma\in S_6$ for which $\sigma \cdot C = C'$ and \[V(\phi_C) = V(\sigma \cdot \phi_C) = V(\phi_{\sigma \cdot C}) = V(\phi_{C'}),\] as desired.
\end{proof}

\begin{remark}
Even though all 15 polynomials $\phi_C$ define the hypersurface $V_{2,6}$ (and so this discussion of cubes and bipartitions did not arise in \cite{CGMS18}), when we turn attention to tropicalization later in this section we will need the extra flexibility provided by the choice of cube $C$.
\end{remark}

For $n > 6$, $V_{2,n}$ is defined scheme-theoretically by the $\binom{n}{6}$ polynomials obtained by pulling $\phi$ back along the projection maps $(\PP^2)^n \rightarrow (\PP^2)^6$ \cite[Theorem 3.6]{CGMS18}.

For $r > 3$ things get trickier; the polynomials found in \cite{CGMS18} were obtained as follows.   The idea is to take the polynomial for $V_{2,6}$, pull it back to $(\PP^2)^{r+3}$, apply the Gale transformation which, up to a constant, in bracket form is simply taking the complement of each index set (see \cite[Proposition 4.5]{CGMS18}) to get a polynomial on $(\PP^{r-1})^{r+3}$, then pull this back to $(\PP^{r-1})^{n}$.  More formally:
\begin{enumerate}
\item Choose $S \in {[n] \choose r+3}$, $T \in {[r+3] \choose 6}$, and a cube $C$ in ${[6] \choose 3}$.
\item Take the pull-back $\pi_{T}^{*}(\phi_{C})$ along the projection $\pi_T : (\PP^2)^{r+3} \rightarrow (\PP^2)^6$.
\item Take the Gale transform $\widehat{\pi_{T}^{*}(\phi_{C})}$.
\item Take the pull-back $\pi_{S}^{*}(\widehat{\pi_{T}^{*}(\phi_{C})})$ along the projection $\pi_S : (\PP^{r-1})^{n} \rightarrow (\PP^{r-1})^{r+3}$. 
\end{enumerate}

In slightly different notation, by using \cite[Proposition 4.1 and Remark 4.2]{CS19} we can rewrite the resulting polynomials explicitly as follows. For each 
\[
	I=\{i_1<\ldots<i_{6}\} \in \binom{[n]}{6}\text { and } J\in\binom{[n]\setminus I}{r-3},
\] 
let $C$ be a cube in $I$ and let $B,W$ be the corresponding bipartition.  For instance, the choice of cube in Example \ref{example:cube} yields
\begin{eqnarray*}
	B = \{\{i_1, i_2, i_3\}, \{i_1, i_4, i_5\}, \{i_2, i_4, i_6\}, \{i_3, i_5, i_6\}\}, \\
	W = \{\{i_1, i_2, i_4\}, \{i_1, i_3, i_5\}, \{i_2, i_3, i_6\}, \{i_4, i_5, i_6\}\}.
\end{eqnarray*}
Then let 
\[
	\psi_{C,I,J} := \prod_{K \in B}m_{J\sqcup K} - \prod_{K \in W}m_{J\sqcup K}.
\]
Each $\psi_{C,I,J}$ vanishes on $V_{r-1,n}$ by \cite[Lemma 4.17]{CGMS18}.

\subsection{Tropical basis}\label{ssec:tropicalbasis}

Since these $\SL_r$-invariant polynomials $\psi_{C,I,J}$ are expressed in bracket form (i.e., they are written as polynomials in the maximal minors) they can immediately be interpreted as polynomial functions on the Grassmannian $\Gr(r, n)$; this is done simply by viewing each minor as the corresponding Pl\"ucker coordinate function.  These are quartic binomials on the Grassmannian, and the choice of cube $C$ corresponds to all 15 possible ways of lifting this to a quartic binomial on the ambient $\PP^{\binom{n}{r}-1}$.   

\begin{definition}\label{def:Srn}
Let $\cS_{r-1, n}$ be the set of bracket binomials $\psi_{C,I,J}$ from \S\ref{ssec:equationsVdn}, and let $\mathcal{J}_{r,n} \subseteq \Bbbk[x_{I}^\pm]_{I\in\binom{[n]}{r}}$ be the ideal generated by $\cS_{r-1, n}$ and the Pl\"ucker relations for $\Gr(r,n)$.
\end{definition}

Note: If $r = 2$ or $r = n-2$, then ${[n] \setminus I \choose r -3} = \emptyset$ for any $I \in {[n] \choose 6}$, so it is safe to extend the preceding definition by setting $\cS_{r-1, n} = \emptyset$ in these cases.

\begin{proposition}\label{prop:ideallocus}
The set-theoretic vanishing locus in $(\Bbbk^\times)^{\binom{n}{r}}$ of the ideal $\mathcal{J}_{r,n}$ is $\varphi_r(\Gr(2,n)^\circ)$.
\end{proposition}

\begin{proof}
The result is trivial for $r=2$, so let $r \ge 3$.  First, we shall establish the set-theoretic containment $\varphi_r(\Gr(2,n)^\circ)\subseteq V(\mathcal{J}_{r,n})$.  By definition the left-hand side is contained in $\Gr(r,n)^\circ$, so all the Pl\"ucker relations vanish on it.  On the other hand, since each $\psi_{C,I,J}$ vanishes on $V_{r-1,n}$, when viewed as a Grassmannian polynomial it vanishes on $\varphi_r(\Gr(2,n)^\circ)$ by Proposition \ref{prop:GelMac}.  So it suffices to establish the reverse set-theoretic containment.

Let $\mathbf{p}\in V(\mathcal{J}_{r,n})$, and let $M(\mathbf{p})$ be any matrix (necessarily with nonzero maximal minors) in the corresponding $\GL_r$-orbit.  Let $\mathbf{p}'$ be any collection of $r+3$ columns in $M(\mathbf{p})$, viewed as a configuration of $r+3$ points in $\PP^{r-1}$.  Due to the nonzero maximal minors, $\mathbf{p}'$ is in general linear position, so we can choose a Gale dual configuration $\mathbf{q}'\in(\mathbb{P}^2)^{r+3}$ and it too is in general linear position by \cite[Proposition 4.5]{CGMS18}.  Now $\psi_{C,I,J}(\mathbf{p}')=0$ for all $I,J$ involving the labels of the points in $\mathbf{p}'$, so by \cite[Theorem~3.6]{CGMS18} $\mathbf{q}'$ must lie on a conic.  This conic must be smooth, since $\mathbf{q}'$ is in general linear position.   It now follows from a classical result of Goppa (see \cite[Corollary 3.2]{EP00}) that the configuration $\mathbf{p}'$ also lies on a rational normal curve, call it $X$.  Now replace a single point of $\mathbf{p}'$ with one of the other columns of $M(\mathbf{p})$ and apply the same argument to deduce that this new configuration lies on a rational normal curve $X'$.  But these two rational normal curves have $r+2$ points in common, so by Castelnuovo's lemma we have $X=X'$.  Repeating this for the remaining columns shows that the full configuration given by $M(\mathbf{p})$ lies on a rational normal curve, and hence $\mathbf{p}\in \varphi_r(\Gr(2,n)^\circ)$ as desired.
\end{proof}

\begin{remark}
We expect that $V(\mathcal{J}_{r,n}) = \varphi_r(\Gr(2,n)^\circ)$ as subschemes of $(\Bbbk^\times)^{\binom{n}{r}}$, not just subvarieties, but we have not been able to establish this.
\end{remark}

By viewing $\psi_{C,I,J}$ as a polynomial on $\mathbb{A}^{\binom{n}{r}}$, we can tropicalize it to obtain a tropical polynomial $\Trop(\psi_{C,I,J})$ on $\mathbb{R}^{\binom{n}{r}}$.  Moreover, since $\psi_{C,I,J}$ is a binomial, the corresponding tropical hypersurface is a classical hyperplane.  Concretely, for coordinates $x_S$ on $\mathbb{R}^{\binom{n}{r}}$, where $S\in\binom{[n]}{r}$, the tropical hypersurface $V^{\trop}(\Trop(\psi_{C,I,J}))$ is given by
\begin{equation}\label{eq:linfcns}
\sum_{K \in B}x_{J\sqcup K} - \sum_{K \in W}x_{J\sqcup K} = 0.
\end{equation}

We first show that the above classical hyperplanes cut out the image of the injective classically-linear map $\Trop(\varphi_r) : \mathbb{R}^{\binom{n}{2}} \hookrightarrow \mathbb{R}^{\binom{n}{r}}$ (recall Lemma \ref{lem:phiinjective}):

\begin{proposition}\label{prop:troplin}
For $2 \le r \le n-2$, we have
\[
	\bigcap_{\psi_{C,I,J}\in \cS_{r-1, n}} V^{\trop}(\Trop(\psi_{C,I,J})) = \Trop(\varphi_r)\left(\mathbb{R}^{\binom{n}{2}}\right).
\]
\end{proposition}

\begin{proof}
If $r = 2$ or $n - 2$, then $\cS_{r-1, n} = \emptyset$ so the left-hand side is $\RR^{{n \choose r}}$, but so is the right-hand side due to the bijectivity of $\Trop(\varphi_r)$ in these cases established in Lemma \ref{lem:phiinjective}. So assume that $3 \le r \le n-3$. 

Let $N$ be the $\left({n \choose 6}\cdot {n-6 \choose r-3}\cdot 15\right) \times {n \choose r}$ matrix whose rows encode the coefficients of the linear forms in \eqref{eq:linfcns}, so that $\ker N$ is the intersection on the left-side of the proposition statement.  Let $M$ be the matrix associated to $\Trop(\varphi_r)$, which was described explicitly in the proof of Lemma~\ref{lem:phiinjective}.  Our task is thus to prove $\ker N = \im M$.  

We shall first show that $NM = 0$, i.e., $\im M \subseteq \ker N$.  From the definition of $M$, this is equivalent to the following: for each $\psi_{C,I,J}$ and each $A\in\binom{[n]}{2}$, the number of terms $x_S$ in \eqref{eq:linfcns} with a positive coefficient for which $A \subseteq S$ equals the number of such terms with a negative coefficient.  If we write the bipartition corresponding to the cube $C$ as \[B = \{B_1,B_2,B_3,B_4\},~W = \{W_1,W_2,W_3,W_4\},\] then the positive terms of $\psi_{C,I,J}$ are $x_{J\sqcup B_j}$ for $j=1,2,3,4$ and the negative terms are $x_{J\sqcup W_j}$ for $j=1,2,3,4$.  So we need to show that the number of $j$ for which $A \subseteq J\sqcup B_j$ equals the number of $j$ for which $A \subseteq J \sqcup W_j$.  This follows immediately from the observations that (1) each element of $I$ occurs in exactly two $B_j$ and two $W_j$, and (2) if a pair of elements of $I$ occurs in a $B_j$ or a $W_j$ then it occurs in exactly one $B_j$ and one $W_j$.  

Having shown that $\im M \subseteq \ker N$, since $\rk M = \binom{n}{2}$ (Lemma~\ref{lem:phiinjective}) it now suffices to show that $\dim(\ker N) \le \binom{n}{2}$, or equivalently, $\rk N \ge {n \choose r} - {n \choose 2}$.  To do this, we will find ${n \choose r} - {n \choose 2}$ linearly independent rows in $N$.  Order the columns of $N$ according to the lexicographic order on $\binom{[n]}{r}$.  We will first find a collection of rows where the left-most nonzero entries all occur in distinct columns, since such rows are necessarily linearly independent, and then we will show that this collection has ${n \choose r} - {n \choose 2}$ elements in it.

Consider a column $I \in \binom{[n]}{r}$.  Let $K=\{a<b<c\}\subseteq I$ be the subset comprising the $3$ smallest elements and let $K^{c} = I \setminus K$ be the remaining $r-3$ elements. Choose another set of $3$ elements $K'=\{d<e<f\}$ in $[n] \setminus I$ satisfying $a < d$, $b <e$, and $c < f$. Consider the following cube $C$ on $K\sqcup K'$:
\[
\{a,b,c\},~\{a,b,f\},~\{a,e,c\},~\{a,e,f\},~\{b,d,c\},~\{b,d,f\},~\{c,d,e\},~\{d,e,f\}.
\]
Notice that $K$ is the smallest vertex of the cube in lexicographic order. Let $B \sqcup W$ be the usual bipartition of $C$, so in particular $K\in B$. Then the vector $(a_{J})$ where 
\[
	a_{J} = \begin{cases}1 & \mbox{if}~J = T \sqcup K^{c},\; T \in B\\
	-1 & \mbox{if}~J = T \sqcup K^{c},\; T \in W\\
	0 &\mbox{otherwise},\end{cases}
\]
is a row of $N$ such that the first nonzero entry is $a_I=1$. Let $p_{n, r}$ be the number of columns $I$ for which we can construct a row $(a_J)$ by the above description. We will show that $p_{n, r} = {n \choose r} - {n \choose 2}$ by using induction on $n$. It is obvious that $p_{r+2, r} = 0$.

Now we count the possibilities. First of all, in $K=\{a<b<c\}$ we have that $a = 1$ or $a > 1$. The number of cases with $a>1$ is precisely $p_{n-1, r}$, which by inductive assumption is equal to ${n-1 \choose r} - {n-1 \choose 2}$. Thus we only have to count the cases with $a = 1$. 

The possible range of $c$ is $3 \le c \le n-r+2$, as we need at least $r-2$ elements in $[n]$ larger than $c$, namely, $K^{c} \cup \{f\}$. When $c < n - r + 2$, then $b$ can be any number between $2$ and $c -1$, so the number of possibilities for $b$ is $c-2$. In this case, the number of ways to choose $K^{c}$ is ${n-c \choose r-3}$. When $c = n - r + 2$, then $b$ cannot be $c-1$, because we need two elements larger than $c$ (for $e$ and $f$) to make a cube where the smallest term is $K$, but in $[n] \setminus K$, there is only one element larger than $b$. So for $b$ we have $c-3=n-r-1$ possibilities and $\binom{n-c}{r-3}=\binom{r-2}{r-3}$ possibilities for $K^c$.

In summary, the number of ways to make such a construction is
\[
	\sum_{c = 3}^{n-r+1}(c-2){n-c \choose r-3} + (n-r-1)\binom{r-2}{r-3} = 
\]
\[
	\sum_{c = 3}^{n-r+2}(c-2){n-c \choose r-3} - {r-2 \choose r-3} = \sum_{i=1}^{n-r}i{n - 2- i \choose r-3} - (r-2).
\]
Thus we obtain a recursive formula 
\[
	p_{n, r} = \sum_{i=1}^{n-r}i{n-2-i \choose r-3} - (r-2) + p_{n-1, r}.
\]
From the inductive assumption and the lemma below, we obtain that $p_{n, r} = {n \choose r} - {n \choose 2}$.
\end{proof}

\begin{lemma}
\begin{equation*}
	\sum_{i=1}^{n-r}i{n-2-i \choose r-3} - (r-2) = \left({n \choose r} - {n \choose 2}\right) - \left({n-1 \choose r} - {n-1 \choose 2}\right).
\end{equation*}
\end{lemma}

\begin{proof}
First, note that $\binom{n}{r} - \binom{n-1}{r} = {n-1 \choose r-1}$ and $\binom{n}{2}-\binom{n-1}{2} = n-1$, so the right-hand side in the formula equals ${n-1 \choose r-1} - (n-1)$. Thus, the identity we need to show is equivalent to 
\[
	\sum_{i=1}^{n-r}i{n -2 - i \choose r-3} = {n-1 \choose r-1} - (n-r+1),
\]
which in turn is equivalent to 
\[
	\sum_{i=1}^{n-r+1}i{n-2-i \choose r-3} = {n-1 \choose r-1}.
\]
By using the substitution $m = n-1$ and $s = r-1$, this is equivalent to 
\[
	\sum_{i=1}^{m-s+1}i{m-1-i \choose s-2} = {m \choose s}.
\]
This last form of the identity can be established by combinatorial considerations: the term $i{m-1-i \choose s-2}$ is precisely the number of ways one can choose a subset of $[m]$ of cardinality $s$ whose second smallest entry is $i+1$. 
\end{proof}

\begin{remark}
The analogue of Proposition~\ref{prop:troplin} for the unweighted dissimilarity map does not hold. As we discussed in Remark~\ref{boccicoolswork}, for the $r$-dissimilarity map we have that $\mathrm{Trop}(\varphi_r)$ is replaced by the Bocci--Cools piecewise linear map $\phi^{(r)}$. Hence, in general $\phi^{(r)}\left(\mathbb{R}^{\binom{n}{2}}\right)$ is not equal to the intersection of tropical hypersurfaces $V^{\mathrm{trop}}(\mathrm{Trop}(\psi_{C,I,J}))$, which is instead a linear subspace of $\mathbb{R}^{\binom{n}{r}}$ as Proposition~\ref{prop:troplin} shows.
\end{remark}

The \emph{three-term Pl\"ucker relations} are the polynomials 
\[x_{ijA}x_{klA} - x_{ikA}x_{jlA} + x_{ilA}x_{jkA}, \]\[ 1 \le i < j < k < l \le n, ~A\in\binom{[n]\setminus\{i,j,k,l\}}{r-2}\]
with the standard convention that index sets are permuted to be increasing and the corresponding permutation signs are included when doing so.  These do not generate the full ideal of Pl\"ucker relations in general, but they do so when passing to the Laurent polynomial ring so they define $\Gr(r,n)^\circ$ in the torus $(\Bbbk^\times)^{\binom{n}{r}}$.  In particular, in Definition \ref{def:Srn} we could have defined the same ideal $\mathcal{J}_{r,n}$ by using only the 3-term Pl\"ucker relations rather than all the Pl\"ucker relations.  

For $r=2$, the three-term Pl\"ucker relations form a tropical basis for the ideal of Pl\"ucker relations \cite[Corollary 4.3.12]{MS15}, which means (1) as already noted, they generate the ideal of Pl\"ucker relations in the Laurent polynomial ring, and (2) the intersection of the tropical hypersurfaces 
\[
	V^{\trop}(\Trop(x_{ij}x_{kl} - x_{ik}x_{jl} + x_{il}x_{jk}))
\]
for $1 \le i < j < k < l \le n$ equals $\Gr^{\trop}(2,n)$ in $\mathbb{R}^{\binom{n}{2}}$.

\begin{theorem}\label{thm:tropicalbasis}
Fix $2 \le r \le n-2$. The three-term Pl\"ucker relations together with the bracket binomials $\psi_{C,I,J}$ form a tropical basis for the ideal $\mathcal{J}_{r,n}$.
\end{theorem}

\begin{proof} 
We already noted above that these polynomials generate $\mathcal{J}_{r,n}$, since we are working in the Laurent polynomial ring.  So, we just need to show that the intersection of the tropical hypersurfaces defined by these polynomials coincides with the tropicalization of the vanishing locus of the ideal $\mathcal{J}_{r,n}$.  By Proposition \ref{prop:ideallocus} we have $V(\mathcal{J}_{r,n}) = \varphi_r(\Gr(2,n)^\circ)$, and by Theorem~\ref{thm:tropim} we have $\Trop(\varphi_r(\Gr(2,n)^\circ)) = d_r^{wt}(\mathcal{T}_n)$.  Thus, our task is reduced to showing that the intersection of the tropical hypersurfaces associated to the polynomials in the theorem statement equals the space of weighted $r$-dissimilarity vectors  $d_r^{wt}(\mathcal{T}_n)$.  

Proposition~\ref{prop:compo} shows that $d_r^{wt}(\mathcal{T}_n) = \Trop(\varphi_r)\left(\Gr^{\trop}(2,n)\right)$, and Proposition \ref{prop:troplin} shows that the intersection of the tropical hypersurfaces associated to the $\psi_{C,I,J}$ is $\Trop(\varphi_r)\left(\mathbb{R}^{\binom{n}{2}}\right)$.  Thus, all that remains is to show that the pull-backs along $\varphi_r$ of the three-term Pl\"ucker relations define $\Gr^{\trop}(2,n)$; indeed, this suffices since $\varphi_r$ is a monomial map hence pulling back along it commutes with tropicalization.  From the definition of $\varphi_r$ we have
\[
	\varphi_{r}^*(x_{ijA}x_{klA} - x_{ikA}x_{jlA} + x_{ilA}x_{jkA}) = \] \[(x_{ij}x_{kl} - x_{ik}x_{jl} + x_{il}x_{jk})\prod_{B \in \binom{A}{2}}x_B^{2}\prod_{t \in A}x_{it}x_{jt}x_{kt}x_{lt},
\]
so in the Laurent polynomial ring the three-term Pl\"ucker relations for $\Gr(r,n)^\circ$ pull back to the three-term Pl\"ucker relations for $\Gr(2,n)^\circ$, which as we noted above are a tropical basis. 
\end{proof}

\begin{remark}\label{rem:reducedtropicalbasis}
It follows from this proof that not all of the bracket binomials $\psi_{C,I,J}$ are needed to form this tropical basis.  Indeed, the only role they play is cutting out the codimension $\binom{n}{r}-\binom{n}{2}$ linear subspace that is the image of $\Trop(\varphi_r)$, so this codimension is the number that is actually needed if they are chosen correctly.  Similarly, not all the three-term Pl\"ucker relations are needed: for each $4$-tuple $i,j,k,l$, only a single choice of $A\in\binom{[n]\setminus\{i,j,k,l\}}{r-2}$ is needed (and any such choice will do).
\end{remark}

As an immediate corollary, by spelling out explicitly the conditions defining the tropical hypersurfaces for each polynomial in this tropical basis we obtain a characterization of weighted dissimilarity vectors, generalizing the classic tree-metric theorem for $2$-dissimilarity vectors:

\begin{corollary}\label{cor:tropicalbasis}
A vector $w=(w_I)_{I\in \binom{[n]}{r}} \in \mathbb{R}^{\binom{n}{r}}$ is a weighted $r$-dissimilarity vector if and only if the following two conditions hold:
\begin{enumerate}

\item for each $4$-tuple $\{i,j,k,l\} \subseteq [n]$ there exists an $A\subseteq [n]\setminus \{i,j,k,l\}$ of size $r-2$ such that two of the following expressions equal each other and are greater than or equal to the third:
\[w_{ijA} + w_{klA},~w_{ikA}+w_{jlA},~w_{ilA} + w_{jkA};\]

\item for each $I \in \binom{[n]}{6}$, $J\in\binom{[n]\setminus I}{r-3}$, and for each cube $C$ in $I$ with corresponding bipartition $B,W$ we have

\[
\sum_{K \in B}w_{J\sqcup K}=\sum_{K \in W}w_{J\sqcup K}.
\]
\end{enumerate}
\end{corollary}

\bibliographystyle{alpha}

\end{document}